\documentclass{article}
\usepackage{amsfonts,amsmath,amssymb,amsthm}
\usepackage{a4wide}
 \newtheorem{theorem}{Theorem}
 \newtheorem{lemma}{Lemma}[section]
 \newtheorem{corollary}[lemma]{Corollary}
 \newtheorem{proposition}[lemma]{Proposition}
 \newtheorem{remark}[lemma]{Remark}

\def\Carre#1#2{\vbox{
   \hrule height .#2pt
   \hbox{\vrule width .#2pt height #1pt \kern #1pt
      \vrule width .#2pt}
   \hrule height .#2pt}}

\def\Om{\Omega}
\def\e{\varepsilon}
\def\D{\mathcal{D}}
\def\H{\mathcal{H}}
\def\E{\mathcal{E}}
\def\R{\mathbb{R}}

\def\H{\mathcal{H}}
\def\LM#1{\hbox{\vrule width.2pt \vbox to#1pt{\vfill \hrule width#1pt
height.2pt}}}
\def\LL{{\mathchoice {\>\LM7\>}{\>\LM7\>}{\,\LM5\,}{\,\LM{3.35}\,}}}
\def\restr{{\LL}}
\def\dom{\textup{dom}\,}

\def\rot{{\rm rot\,}}

\def\Div{\textup{div}\,}
\def\d{{\rm d}}
\def\ov{\overline}

\def\half{\frac{1}{2}}

\def\argmin{\textup{arg\,min}}

\def\wtos{\stackrel{*}{\rightharpoonup}}
\begin{document}

\title{On the gradient flow of a one-homogeneous functional}

\author{Ariela Briani  
\thanks{\it LMPT, F\'ed\'eration Denis Poisson, Universit\'e Fran\c cois Rabelais, CNRS, Parc de Grandmont, 37200 Tours, France. 
ariela.briani@lmpt.univ-tours.fr} 
\and Antonin Chambolle
\thanks{\it
CMAP, Ecole Polytechnique, CNRS, 91128 Palaiseau, France. \newline 
antonin.chambolle@cmap.polytechnique.fr}
\and Matteo Novaga
\thanks{\it
Dipartimento di Matematica, Universit\`a di Padova, via Trieste 63, 35121 Padova, Italy. \newline novaga@math.unipd.it}
\and Giandomenico Orlandi
\thanks{\it
Dipartimento di Informatica, Universit\`a di Verona, strada le Grazie 15, 37134 Verona, Italy. \newline giandomenico.orlandi@univr.it}
} 

\date{}

\maketitle

\begin{abstract}
We consider the gradient flow of a one-homogeneous functional, whose
dual involves the derivative of a constrained scalar function.
We show in this case that the gradient flow is related to a weak,
generalized formulation of a Hele-Shaw flow. The equivalence follows
from a variational representation, which is a variant of well-known
variational representations for the Hele-Shaw problem.
As a consequence we get existence and uniqueness of a weak solution to
the Hele-Shaw flow.
We also obtain an explicit representation for the Total Variation
flow in one dimension, and easily deduce basic qualitative properties,
concerning in particular the ``staircasing effect''.
\end{abstract}

\section{Introduction}
This paper deals with the $L^2$-gradient flow of the functional
\begin{eqnarray*}
J_k(\omega) := \int_A |\d \omega|\,dx \qquad k\in \{0,\ldots, N-1\}
\end{eqnarray*}
defined on differential forms $\omega\in L^2(A,\Om^k(\R^N))$, where $A\subseteq\R^N$ is an open set. We will focus on the particular case $k=N-1$: in
that case, the dual variable is a scalar and this yields very particular
properties of the functional $J_k$ and the associated flow.

Notice that, when $k=0$, the functional $J_0$
reduces to the usual total variation.
When $k=N-1$ 
we can identify by duality $\omega\in L^2(A,\Om^{N-1}(\R^N))$
with a vector field $u\in L^2(A,\R^N)$, so that $J_{N-1}$
is equivalent to the functional 
\begin{equation}\label{divu}
\D(u) := \int_A|\Div u|\,dx 
\end{equation}
that is, the total mass of $\Div u$ as a measure. 

The gradient flow of $\D$ has interesting properties:
we show in particular that it is equivalent to a constrained variational
problem for a function $w$ such that $\Delta w = \Div u$. 
Moreover, under some regularity assumption on the initial datum $u_0$,
such a variational problem allows to define a
weak formulation of the Hele-Shaw flow~\cite{ElJa,Gu} 
(see also \cite{KM} for a viscosity formulation).
Therefore, it turns out that the flow of \eqref{divu} provides a
(unique) global weak solution to the Hele-Shaw flow, for a suitable
initial datum $u_0$.
But our formulation allows us to consider quite general initial data $u_0$,
for which for instance  $\Div u_0$ may change sign, or be a measure.

The plan of the paper is the following:
in Section \ref{secgrad} we introduce the general functional we are
interested in, we write the Euler-Lagrange equation for its
Moreau-Yosida approximation and, in Section \ref{secdual},
we express it in a dual form that will be  the base of our analysis. 

In Section \ref{seccase} we focus on the case $k=1$ which is analyzed
in this paper. We show many interesting properties of the flow: comparison,
equivalence with a weak Hele-Shaw flow if the initial datum is 
smooth enough, and qualitative behavior when the initial datum is not
smooth. In Section \ref{secanti} we observe that, in dimension $2$,
the case $k=N-1$ also covers the flow of the $L^1$-norm of the rotation
of a vector field, which appears as a particular limit of the Ginzburg-Landau model
(see~\cite{Orlandietal,SS} and references therein).

Another interesting consequence of  our analysis is that it yields
simple but original qualitative results on the solutions of the
Total Variation flow in dimension one (see also~\cite{BeCaNo,BoFi}). We
show in Section \ref{secone}  that the denoising of a noisy signal
with this approach will, in general, almost surely produce a
solution which is ``flat'' on a dense set. This undesirable artefact
is the well-known ``staircasing'' effect of the Total Variation regularization
and is the main drawback of this approach for signal or image
reconstruction.

%%%%%%%%%%%%%%%%%%%%%%%%%%%%%%%%%%%%%%%%%%%%%%%%%%%%%%%
\section{Gradient flow}\label{secgrad}
%%%%%%%%%%%%%%%%%%%%%%%%%%%%%%%%%%%%%%%%%%%%%%%%%%%%%%%
Given an initial datum $\omega_0\in L^2(A,\Om^k(\R^N))$, 
the general theory of \cite{Brezis} guarantees the existence of a global weak solution 
$\omega\in L^2([0,+\infty),L^2(A,\Om^k(\R^N)))$ of the gradient flow equation of $J_k$:
\begin{equation}\label{eqevol}
\omega_t \in -\partial J_k(\omega) \qquad t\in [0,+\infty)\,,
\end{equation}
where $\partial J_k$ denotes the subgradient of the convex functional $J_k$.
Given $\e>0$ and $f\in L^2(A,\Om^k(\R^N))$, we consider the minimum problem 
\begin{equation}\label{mainfun}
\min_{\omega:A\to \R^N}J_k(\omega)\,+\, \int_A \frac{1}{2\e}|\omega- f|^2\, dx .
\end{equation}
Notice that
\[
\min_{\omega:A\to \R^N}J_k(u) + \int_A\frac{|\omega- \e f|^2}{2}\, dx 
= \e \min_{\omega:A\to \R^N}J_k(u) + \int_A\frac{1}{2\e}|\omega- f|^2\, dx .
\]
The Euler-Lagrange equation corresponding to \eqref{mainfun}
is 
\[
\e (f-\omega) \in \partial J_k(\omega),
\]
that is there exists a $(k+1)$-form $v$ with $|v|=1$ such that $v= \d \omega/|\d \omega|$ 
if $\d \omega\ne 0$, and
\begin{equation}\label{eqeqdual}
\e (f-\omega) = \d^*v \ {\rm in\ }A
\qquad {\rm and} \qquad (*v)_T=0
\ {\rm on\ }\partial A.
\end{equation}

%%%%%%%%%%%%%%%%%%%%%%%%%%%%%%%%%%%%%%%%%%%%%%%%%%%
\subsection{Dual formulation}\label{secdual}
%%%%%%%%%%%%%%%%%%%%%%%%%%%%%%%%%%%%%%%%%%%%%%%%%%%
Equation \eqref{eqeqdual} is equivalent to 
\[
\omega \in \partial J_k^*(\e(f-\omega)),
\]
where 
\[
J_k^*(\eta) := \sup_{w:A\to \R^N} \int_A \eta\cdot w\, dx - J_k(w)
= \left\{ \begin{array}{ll}
0 & {\rm if\ }\| \eta\|_*\le 1
\\
+\infty & {\rm otherwise}
\end{array}
\right.
\]
and
\[
\|\eta\|_* = \sup\left\{ \int_A\eta\cdot w\, dx:\ J_k(w)\le 1\right\}.
\]
Note that 
\[
J_k(w) + J_k^*(\eta) \ge \int_A w\cdot \eta\,dx
\]
for all $w,\eta$. The equality holds iff
$\int_A\eta\cdot w\, dx= J_k(w)$, and in such case we have $\|\eta\|_*\le 1$.

Letting $u$ be a minimizer of \eqref{mainfun} and $\eta=(f-u)/\e$
we then get 
\[
\int_A u\cdot \frac{f-u}{\e}\,dx = J_k(u),
\]
which implies 
\[
J_k^*\left(\frac{f-u}{\e}\right)=0 \qquad {\rm that\ is} \qquad
\e\ge \|f-u\|_*\,.
\]
In particular, we showed the following (see also \cite{Meyer} for
the same result in the case of the Total Variation).

\begin{proposition}
The function $u=0$ is a minimizer of \eqref{mainfun} if and only if
\begin{equation}\label{stimaeps}
\e \ge \e_c := \|f\|_*\,.
\end{equation}
\end{proposition}
Note that $\|\eta\|_*<\infty$ implies that 
\[
\int_A \eta w =0
\]
for all $w$ such that $\d w=0$. 
%In particular, we can restrict the supremum in 
%the definition of $\|\eta\|_*$ to the $w$ such that $\d w\ne 0$, that is 
%we can assume $w=\d^*\beta$ for some $2$-form $\beta$.
By Hodge decomposition, this implies that $\eta =\d^*g$ for some $2$-form $g$, 
with $g_N=0$ on $\partial A$.
It follows that 
\begin{equation}\label{eqpsi}
\|\eta\|_* 
= 
\sup_{\int_A |\d w| \le 1} \int_A \d^* g\cdot w\, dx
=
\sup_{\int_A |\d w| \le 1} \int_A g\cdot \d w\, dx + \int_{\partial A} w \wedge *g_N
= \sup_{\int_A |\d w| \le 1} \int_A g\cdot \d w\, dx.
\end{equation}
We then get
\[
\|\eta\|_* 
= \inf_{\underset{g_N|_{\partial A}=0}{\d^*g=\eta}} \, \| g\|_{L^\infty(A)}.
\]
Indeed, it is immediate to show the $\le$ inequality.
On the other hand, by Hahn-Banach Theorem, there exists a form $g'$, 
with $\d^*g'=\d^* g=\eta$ such that 
\[
\|\eta\|_* 
= 
\sup_{\int_A |\d w| \le 1} \int_A g\cdot \d w\, dx 
= \sup_{\int_A |\psi| \le 1} \int_A g'\cdot \psi\, dx
= \| g\|_{L^\infty(A)}.
\]
Fix now $\phi_0$ such that $\d^*\phi_0 =\eta$. 
We can write $g= \phi_0+\d^*\psi$, so that \eqref{eqpsi} becomes
\begin{equation}\label{eqinfty}
\|\eta\|_* = \min_{\psi:\,(\phi_0+\d^*\psi)\cdot\nu^{}_{A}=0}\| \phi_0+\d^*\psi\|_{L^\infty(A)}.
\end{equation}
The Euler-Lagrange equation of \eqref{eqinfty} is similar to the  
\emph{infinity laplacian equation}
\[
{\rm d}_\infty (\phi_0+\d^*\psi) = 0.
\]
By duality problem \eqref{eqinfty} becomes
\begin{equation}\label{eqinftytre}
\min_{\psi\in W^{1,\infty}_0(A)}\| \nabla \psi + \phi_0\|_{L^\infty(A)},
\end{equation}
and the corresponding Euler-Lagrange equation is
\begin{equation}\label{ELtre}
\langle (\nabla^2\psi + \nabla\phi_0)(\nabla\psi +\phi_0),(\nabla\psi +\phi_0)\rangle = 0.
\end{equation}

%%%%%%%%%%%%%%%%%%%%%%%%%%%%%%%%%%%%%%%%%%%%%%%%%%%%%%%%%%%%%%%
\section{The case $k=N-1$}\label{seccase}
%%%%%%%%%%%%%%%%%%%%%%%%%%%%%%%%%%%%%%%%%%%%%%%%%%%%%%%%%%%%%%%

In this case, we recall that we are considering the gradient flow
of the functional~\eqref{divu}, which is defined, for any
$u\in L^1_{\textup{loc}}(A;\R^N)$, as follows
\begin{equation}\label{dualdivu}
\D(u)\ =\ \sup\left\{
\int_A -u\nabla v\,dx\ :\ v\in C_c^\infty(A)\,, |v(x)|\le 1\ \forall
x\in A\right\}.
\end{equation}
This is finite if and only if the distribution $\Div u$ is
a bounded Radon measure in $A$. We now see it as a (convex,
l.s.c., with values in $[0,+\infty]$) functional over the Hilbert space
$L^2(A;\R^N)$: it is then clear from~\eqref{dualdivu} that
it is the support function of
\[
K\ =\ \left\{ -\nabla v\,:\, v\in H^1_0(A;[-1,1])\right\}
\]
and in particular $p\in \partial \D(u)$, the subgradient
of $\D$ at $u$, if and only if $p\in K$ and
$\int_A p\cdot u\,dx=\D(u)=\int_A |\Div u|$:
\[
\partial \D(u)\ =\ \left\{
-\nabla v\,:\, v\in H^1_0(A;[-1,1])\,,\ \int_A -\nabla v\cdot u\,dx
\,=\, \int_A |\Div u| \right\}\,.
\]

% In this case, $p=-\nabla v$ for some $v\in H^1_0(A;[-1,1])$ and
% we can find $v_n\in C_c^\infty(A;[-1,1])$ which
% converges strongly to $v$ in $H^1_0(A)$, and quasi-everywhere
% (in the sense of the $H^1$-capacity, see~\cite{AdamsHedberg}), and we have
% \[
% \int_A v_n\Div u\ \to \ \int_A |\Div u|.
% \]
% We deduce that $(\Div u)^+$ is supported by the set $\{v=1\}$,
% while $(\Div u)^-$ is supported by the set $\{v=-1\}$. We observe
% that when $u\in L^2(A;\R^N)$, the measure $\Div u\in H^{-1}(A)$
% must vanish on
% sets of $H^1$-capacity $0$~\cite[\S 7.6.1]{AdamsHedberg},
% so that this characterization makes sense.

We can define, for $u\in \dom\D$, the Radon-Nikodym density
\[
\theta_{\Div u}(x)\ =\ \frac{\Div u}{|\Div u|}(x)
\ =\ \lim_{\rho\to 0} \frac{\int_{B(x,\rho)}\Div u}{\int_{B(x,\rho)}|\Div u|},
\]
which exists $|\Div u|$-a.e.~(we consider that it is defined only when
the limit exists and is in $\{-1,1\}$), and is such that
$\Div u=\theta_{\Div u} |\Div u|$. We can also introduce the
Borel sets
\[
\E^\pm_u\ =\ \left\{ x\in A\,:\, \theta_{\Div u}(x)=\pm 1\right\}\,.
\]
Then, we have:
\begin{lemma}\label{characsubgrad}
\[
\partial\D(u)\ =\ \left\{-\nabla v\,:\, v\in H^1_0(A;[-1,1])\,,
v=\pm 1\ \textup{ $| \Div u|-$ a.e.~on}\ \E^\pm_u\right\}\,.
\]
%Moreover, the {\emph precise representative} of $v$ is $\pm 1$
%everywhere on $\E^\pm_u$. /// not true
\end{lemma}
%Recall that the precise representative of $v$ is defined by
%\[
%\tilde{v}(x)\ =\ \lim_{\rho\to 0} \frac{\int_{B(x,\rho}v(y)\,dy}{|B(x,\rho)|}
%\]
%and that this limit exists quasi-everywhere in $A$.
\begin{proof}
Consider  $v\in H^1_0(A;[-1,1])$. Then we know~\cite{AdamsHedberg}
that it is the limit of smooth functions $v_n\in C_c^\infty(A;[-1,1])$
with compact support which converge to $v$ quasi-everywhere (that is,
up to a set of $H^1$-capacity zero).

We recall that when $u\in L^2(A;\R^N)$, the measure $\Div u\in H^{-1}(A)$
must vanish on sets of $H^1$-capacity $0$~\cite[\S 7.6.1]{AdamsHedberg}:
it follows that
$v_n\to v$ $|\Div u|$-a.e.~in $A$. Hence, by Lebesgue's convergence
theorem,
\[
-\int_A \nabla v(x)u(x)\,dx
\ =\ \lim_{n\to \infty}\int_A v_n(x)\theta_{\Div u}(x)|\Div u|(x)
\ =\ \int_A v(x)\Div u (x)\,.
\]
%The lemma easily follows.
It easily follows that if $v=\pm 1$ $| \Div u |$- a.e.~on $\E^{\pm}_u$,
$-\nabla v\in\partial\D(u)$ and conversely, that
if $v\in\partial\D(u)$ then $v=\pm 1$ $|\Div u|$-a.e.~on $\E^{\pm}_u$.
% We should now show that when $v$ is a subgradient, we
% have $v=\theta_{\Div u}$ q.e., and not only $|\Div u|$-a.e.
% \problem{BUT HOW?? IDEA: $\{v< 1-\e\}$ is quasi (and finely)-open
% and one should show that it cannot contain a set with positive capacity
% where $\theta=1$... MAY BE FALSE}
\end{proof}

We now define, provided $u\in\dom\partial \D$
(i.e.,~$\partial \D(u)\neq \emptyset$),
\[
\partial^0 \D(u)
\ =\ \argmin\left\{ \int_A |p|^2\,dx\,:\, p\in \partial \D(u)\right\}\,:
\]
it corresponds to the element $p=-\nabla v\in \partial \D(u)$
of minimal $L^2$-norm. Using Lemma~\ref{characsubgrad}, equivalently,
$v$ is the function which minimizes $\int_A |\nabla v|^2\,dx$ among
all $v\in H^1_0(A)$ with $v\ge \chi_{\E^+_u}$ and $v\le -\chi_{\E^-_u}$,
 $|\Div u|$-a.e.: 
%\problem{(OR $|\Div u|$-a.e?)}  
 in particular, we deduce that
it is harmonic in $A\setminus \ov{\E^+_u\cup \E^-_u}$.

Let us now return to the flow~\eqref{eqevol}. In this setting,
it becomes
\begin{equation}\label{eqdiv}
\left\{\begin{array}{ll}
u_t &=  \nabla v
\\
u(0) &= u_0
\end{array}\right.
\end{equation}
where $v$ satisfies $|v|\le 1$ and 
\[
\D(u)+\int_A u\cdot \nabla v =0.
\]
It is well know, in fact, that the solution of~\eqref{eqdiv}
is unique and that $-\nabla v(t)= \partial^0\D(u(t))$ is the
right-derivative of $u(t)$ at any $t\ge 0$~\cite{Brezis}.
Given the solution $(u(t),v(t))$ of \eqref{eqdiv}, we let
\[
w(t)\ :=\ \int_0^t v(s)\,ds,
\]
which takes its values in $[-t,t]$. We have 
\[
u(t)= u_0+\nabla w(t).
\]
\begin{theorem}
Assume $u_0\in L^2(A;\R^N)$.
The function $w(t)$ solves the following obstacle problem
\begin{equation}\label{obst}
\min \left\{ \half\int_A |u_0+\nabla w|^2\,dx\,:\,
w\in H^1_0(A)\,, |w|\le t \textrm{ a.e.}\right\}.
\end{equation}
\end{theorem}

Observe that in case we additionally have
$\Div u_0\ge \alpha> 0$, this obstacle problem
is well-known for being an equivalent formulation of the Hele-Shaw problem, see~\cite{ElJa,Gu}.

\begin{proof}
Given $u_{0}\in L^2(A;\R^N)$, we can recursively
define $u_{n+1}\in L^2(A;\R^N)$ as the unique solution of the minimum problem
\[
\min_{u\in L^2(A,\R^N)} \D_\e(u,u_{n}),
\]
where 
\[
\D_\e(u,v)= \D(u)\,+\, \int_A \frac{1}{2\e}|u- v|^2\, dx .
\]
Then, there exists $v_{n+1}\in\partial \D(u_{n+1})$ such that
\begin{equation}\label{eqvn}
u_{n+1}-u_n-\e \nabla v_{n+1}=0.
\end{equation}
It follows that $v_{n+1}\in H^1_0(A)$ minimizes the functional
\[
\int_A |u_n + \e\nabla v|^2\,dx
\]
under the constraint $|v|\le 1$. Let now 
\[
w_n := \e\sum_{i=1}^n v_i.
\]
The from \eqref{eqvn} we get
\begin{equation}\label{eqwn}
u_n=u_0 + \nabla w_{n},
\end{equation}
and $w_n$ minimizes the functional
\begin{equation}\label{eqfunw}
\int_A |u_0 + \nabla w|^2\,dx
\end{equation}
under the constraint $|w-w_{n-1}|\le \e$. 
Notice that $|w_n-w_{n-1}|\le \e$ for all $n$ implies 
\begin{equation}\label{conw}
|w_n|\le n\e.
\end{equation}
We now show that $w_n$ minimizes \eqref{eqfunw} also under the 
weaker constraint \eqref{conw}.
Indeed, letting $\hat w_n$ be the minimizer of \eqref{eqfunw} 
under the constraint \eqref{conw}, we have
\[
\hat w_n-\e\le \hat w_{n+1}\le \hat w_n+\e,
\] 
which follows by noticing that $\min\{ \hat w_n,\hat w_{n+1}+\e\}$
and $\max\{ \hat w_n,\hat w_{n+1}-\e\}$ minimize \eqref{eqfunw}, hence 
they are both equal to $\hat w_n$.
It then follows $w_n = \hat w_n$ for all $n$. 

\smallskip

Passing to the limit in $n$ we get the corresponding result in the continuum case.
\end{proof}

\begin{remark}\label{remsemigroup}\textup{
The previous proof also shows that for any initial $u_0\in L^2(A;\R^N)$,
$u(t)=u_0 + \nabla w(t)$ is the unique minimizer of 
$$
\int_A |\Div u| + \frac{1}{2t}|u-u_0|^2\,dx.
$$
We recall
that obviously,
such property does not hold for general semigroups generated by the
gradient flow of a convex function. It is shown in \cite{AlterCasellesChambolle}
to be the case for the Total Variation flow, in any dimension, when
the initial function is the characteristic of a convex set.}
\end{remark}

\medskip

\subsection{Some properties of the solution}

A first observation is that $t\mapsto w(t)$ is continuous (in 
$H^1_0(A)$, strong), as follows both from the study of the
varying problems~\eqref{obst} and from the fact that the
flow $u(t)=u_0+\nabla w(t)$ is both continuous at zero
and $L^2(A)$-Lipschitz continuous away from $t=0$ (and up
to $t=0$ if $u_0\in \dom\partial \D$).

In fact, one can check that $w$ is also $L^\infty$-Lipschitz
continuous in time: indeed, it follows from the comparison principle
that for any $s\le t$,
\begin{equation}\label{wlip}
w(s)-t+s\ \le \ w(t)\ \le\ w(s)+t-s
\end{equation}
a.e.~in $A$, hence $\|w(t)-w(s)\|_{L^\infty(A)}\le |t-s|$.
The comparison~\eqref{wlip} is obtained by adding the
energy in~\eqref{obst} of $\min\{w(t),w(s)+t-s\}$ (which is
admissible at time $t$ and hence should have an energy larger
than the energy of $w(t)$) to the energy of $\max\{w(t)-t+s,w(s)\}$
(which is admissible at time $s$), and checking that this sum is
equal to the energy at time $t$ plus the energy at time $s$.
This is quite standard, see~\cite{Caf,KM}.

In particular, we can define for any $t$ the sets
\begin{equation}\label{defEpm}
E^+(t)\ =\ \{ \tilde{w}(t)=t\}
\ \textup{ and }\ E^-(t)\ =\ \{\tilde{w}(t)=-t\}\,,
\end{equation}
where $\tilde w(t)$ is the precise representative of $w(t)\in H^1(A)$,
 defined quasi-everywhere
by
\begin{equation}\label{precise}
\tilde w(t,x) 
\ =\ \lim_{\rho\to 0} \frac{1}{\omega_N\rho^N}\int_{B(x,\rho)}w(t,y)dy
\end{equation}
($\omega_N$ is the volume of the unit ball). It follows
from~\eqref{wlip} and~\eqref{precise}
that if $\tilde{w}(t,x)=t$, then for any $s<t$,
$x$ is also a point where $\tilde{w}(s,x)$ is well-defined, and its
value is $s$; similarly if $\tilde{w}(t,x)=-t$ then $\tilde{w}(s,x)=-s$.
Hence: the functions $t\mapsto E^+(t)$,  $t\mapsto E^-(t)$ are nonincreasing.

Also, if $s<t$, one has from~\eqref{wlip}
\[
\frac{1}{\omega_N\rho^N}\int_{B(x,\rho)}w(s,y)dy\,-t+s
\le \frac{1}{\omega_N\rho^N}\int_{B(x,\rho)}w(t,y)dy
\le
\frac{1}{\omega_N\rho^N}\int_{B(x,\rho)}w(s,y)dy\,-s+t
\]
so that if $x\in E^+(s)$,
\[
2s-t\ \le\ \liminf_{\rho\to 0} \frac{1}{\omega_N\rho^N}\int_{B(x,\rho)}w(t,y)dy
\ \le\ \limsup_{\rho\to 0} \frac{1}{\omega_N\rho^N}\int_{B(x,\rho)}w(t,y)dy
\ \le\ t
\]
and sending $s$ to $t$, we find that if $x\in \bigcap_{s<t} E^+(s)$,
$\tilde{w}(t,x)=t$ and $x\in E^+(t)$: hence these sets (as well as
$E^-(\cdot)$) are left-continuous.

We define
\begin{equation}\label{defEdpm}
E_r^+(t)\ =\ \bigcup_{s>t} E^+(s)\ \subseteq E^+(t)
\ \textup{ and }\ E_r^-(t)\ =\ \bigcup_{s>t} E^-(s)\ \subseteq E^-(t)\,,
\end{equation}
as well as $E(t)=E^+(t)\cup E^-(t)$, $E_r(t)=E_r^+(t)\cup E_r^-(t)$.
Then, there holds the following lemma:
\begin{lemma}\label{comparE}
If $s\le t$, then
\[
E^-(t)\ \subseteq E^-(s) \textup{ and } E^+(t)\ \subseteq E^+(s),
\]
\[
E_r^-(t)\ \subseteq E_r^-(s) \textup{ and } E_r^+(t)\ \subseteq E_r^+(s).
\]
Moreover, for $t>0$, $v(t)=\pm 1$ quasi-everywhere on $E_r^\pm(t)$
and $\E^\pm_{u(t)}\subseteq E_r^\pm(t)$, up to a set  $|\Div u(t)|$-negligible.
In particular 
$$\Div u(t)\restr (E_r^-(t))^c\ge 0, \quad 
\Div u(t)\restr (E_r^+(t))^c\le 0, \quad
\Div u(t)\restr (E_r^+(t)\cup E_r^-(t))^c= 0.
$$

%\problem{[??] or $|\Div u(t)|$-negligible}.
\end{lemma}
Here, for a Radon measure $\mu$ and a Borel set $E$, $\mu\restr E$ denotes the
measure defined by $\mu\restr E(B):=\mu(E\cap B)$. 
\begin{proof}
The first two assertions, as already observed, follow from~\eqref{wlip}
and the definition of $E^\pm_r$.
We know that the solution of equation~\eqref{eqdiv} satisfies
$\partial^+_t u=-\partial^0\D(u(t))=\nabla v(t)$ for any $t>0$, but the
right-derivative of $u=u_0+\nabla w(t)$ is nothing else as
$\lim_{h\to 0} \nabla [w(t+h)-w(t)]/h$. We easily deduce
that $v(t)=\lim_{h\to 0}[w(t+h)-w(t)]/h$ (which converges in $H^1_0$-strong).
Since when $x\in E_r^+(t)$, $\tilde w(t,x)=t$ and $\tilde w(t+h,x)=t+h$ 
for $h$ small enough,
we deduce that $v(x)=1$ on that set, in the same way $v=1$ on
$E_r^-(t)$.\medskip

Observe that the Euler-Lagrange equation for~\eqref{obst} is
the variational
inequality
\[
\int_A (u_0+\nabla w(t))\cdot(t\nabla v-\nabla w(t))\,dx\ \ge \ 0\,,
\]
for any $v\in H^1_0(A;[-1,1])$. In other words since
$u(t)=u_0 + \nabla w(t)$,
\[
-\int_A u(t)\cdot\nabla \frac{w(t)}{t}\ \ge\ -\int_A u(t)\cdot\nabla v
\]
for any $|v|\le 1$,
and we recover that $-\nabla w(t)/t\in \partial\D(u(t))$.

Hence (using Lemma~\ref{characsubgrad}),
$\E_{u(t)}^\pm\subseteq E^\pm(t)$. Now, if $\tilde{v}\in H^1_0(A;[-1,1])$
with $\tilde{v} =\pm 1 $ on $E_r^\pm$, one deduces that for any
$s>t$,
\[
-\int_A \nabla \tilde{v}\cdot u(s)\,dx\ =\ \D(u(s))\,.
\]
Sending $s\to t$, it follows
\[
-\int_A \nabla \tilde{v}\cdot u(t)\,dx\ \ge\ \D(u(t))\,,
\]
hence $\tilde{v}\in  \partial \D(u(t))$. We deduce that
$\E^\pm_{u(t)}\subseteq E_r^\pm(t)$, invoking Lemma~\ref{characsubgrad}.
\end{proof}

%% \noindent\textbf{Remark.}
%% Observe that the finite difference quotient $v_h=(w(t)-w(t-h))/h$ is
%% also the solution of the minimum obstacle problem
%% \begin{equation}\label{obsth}
%% \min \left\{ \half\int_A |u(t-h)+h\nabla v|^2\,dx\,:\,
%% v\in H^1_0(\Om)\,, |v|\le 1 \textrm{ a.e.}\right\}\,.
%% \end{equation}
%% Hence, in the same way we get
%% that $-\nabla v_h\in \partial\D(u(t))$ (which
%% passes to the (weak) limit since $\partial\D(u(t))$ is closed
%% and convex).\medskip

\begin{remark}\textup{We might find situations where $|v(t)|=1$ outside
of the contact set. For instance, assume the problem is radial,
$\Div u_0$ is positive in a crown and negative in the center. Then
one may have that $E^+$ is a crown ($w$ should be less than $t$
at the center) and $E^-$ is empty. In that case, $v$ should be
equal to one also in the domain surrounded by the set $E^+$.}
\end{remark}
\medskip

We show now another simple comparison lemma:
\begin{lemma}\label{lemcomparw}
Let $u_0$ and $u'_0$ in $L^2(A;\R^N)$ such that
\[
\Div u'_0\ \le\ \Div u_0
\]
in $H^{-1}(A)$. Then for any $t\ge 0$, $w'(t)\le w(t)$, where
$w'(t)$ and $w(t)$ are the solutions of the contact problem~\eqref{obst},
the first with $u_0$ replaced with $u'_0$.
\end{lemma}

\begin{proof}
Let $t>0$, $\e>0$, and $w^\e$ be the minimizer of
\[
\min_{|w|\le t}\half \int_A |\nabla w|^2\,dx \,-\,\int_A w(\Div u'_0-\e)
\]
which of course is unique. We now show that $w^\e\le w(t)$ a.e., and
since $w^\e\to w'(t)$ as $\e\to 0$ the thesis will follow.

We have by minimality
\begin{equation*}
\begin{array}{l}
\displaystyle
 \int_A  \frac{|\nabla w(t)|}{2}^2 dx\, -\, \int_A w(t)(\Div u_0)
\, \le\, \int_A \frac{|\nabla (w(t)\vee w^\e)|}{2}^2 dx
\,-\, \int_A (w(t)\vee w^\e)(\Div u_0)\,,
\\[5mm]
\displaystyle
\int_A \frac{|\nabla w^\e|}{2}^2 dx \,-\,\int_A w^\e(\Div u'_0-\e)
\, \le\,  \int_A \frac{|\nabla (w(t)\wedge w^\e)|}{2}^2 dx
\,-\,\int_A (w(t)\wedge w^\e)(\Div u'_0-\e)\,,
\end{array}
\end{equation*}
where we denote $w(t)\vee w^\e:=\max\{w(t),w^\e\}$ and $w(t)\wedge w^\e :=\min\{w(t),w^\e\}$.
Summing both inequalities we obtain
\[
\int_A (w(t)\vee w^\e -w(t))\Div u_0
\ \le\ \int_A (w^\e-w(t)\wedge w^\e)(\Div u'_0 -\e)\,,
\]
from which it follows $\e\int_A (w^\e-w(t))^+\,dx\le 0$, which
is our claim.
\end{proof}

\begin{corollary}
Under the assumptions of Lemma~\ref{lemcomparw},
%% we obtain that if $\Div u'_0\ \le\ \Div u_0$:
\begin{equation}\label{comparsets}
E^-(t)\ \subseteq {E'}^-(t)\textrm{ and }{E'}^+(t)\ \subseteq \ E^+(t)\,,
\end{equation}
and it follows that $v'(t)\le v(t)$, for each $t> 0$.
\end{corollary}
\begin{proof} Eqn~\eqref{comparsets} follows at once from
the inequality $w'(t)\le w(t)$ (Lemma~\ref{lemcomparw}).
We deduce, of course, that also $E_r^-(t)\ \subseteq {E'}_r^-(t)$,
and ${E'}_r^+(t)\ \subseteq \ E_r^+(t)$.
Consider the function $v=v'(t)\wedge v(t)=\min\{v'(t),v(t)\}$.
As it is $\pm 1$ on ${E'}_r^\pm(t)$, it follows from Lemmas~\ref{comparE}
and~\ref{characsubgrad} that $-\nabla v\in \partial\D(u'(t))$. In the same
way, $v'=v'(t)\vee v(t)=\max\{v'(t),v(t)\}$ is such that
$-\nabla v'\in \partial \D(u(t))$. Since
\[
\int_A |\nabla v|^2\,dx+\int_A |\nabla v'|^2\,dx
\ =\ \int_A |\nabla v(t)|^2\,dx+\int_A |\nabla v'(t)|^2\,dx\,,
\]
either $\int_A |\nabla v|^2\,dx\le \int_A |\nabla v(t)|^2\,dx$
or $\int_A |\nabla v'|^2\,dx\le \int_A |\nabla v'(t)|^2\,dx$.
By minimality (as $-\nabla v(t)=\partial^0\D(u(t))$)
it follows that $v=v(t)$ and $v'=v'(t)$.
\end{proof}

%In the next section we will show that, provided $\Div u_0$ and $\Div u'_0$
%are regular enough, a similar inequality holds for these quantities.

%%%%%%%%%%%%%%%%%%%%%%%%%%%%%%%%%%%%%%%%%%%%%%%%%%%%%%%%%%%%%%%%%

\subsection{The support of the measure $\Div u$}

Throughout this section we will assume that $\Div u_0$ is a bounded
Radon measure on $A$.

\begin{lemma}\label{lemmeasures}
Let $u_0\in L^2(A;\R^N)\cap \dom \D$, $\delta>0$ and
$u=(I+\delta\partial \D)^{-1}(u_0)$. Then for a positive
Radon measure $\mu\in H^{-1}(A)$, the Radon-Nikodym derivatives of
$\Div u$ and $\Div u_0$ with respect to $\mu$ satisfy
$(\Div u/\mu) (x)\le (\Div u_0/\mu) (x)$
for $\mu$-a.e.~$x\in\E^+_u$, and
$(\Div u/\mu)(x)\ge (\Div u_0/\mu)(x)$ for $\mu$-a.e.~$x\in \E^-_u$.
In particular, $\Div u << \Div u_0$ and
$(\Div u)^\pm \le (\Div u_0)^\pm$.
%and if $u_0\in \dom\partial J$,
%$E^\pm[\psi]\subseteq E^\pm[\psi^0]$
\end{lemma}
%\noindent\textbf{Remark.}
\begin{remark}\label{Remdensite}\textup{
It follows from the Lemma that
$\Div u=\theta\Div u_0\restr(\E^+_u\cup \E^-_u)$ for some
weight $\theta(x) \in [0,1]$. We can build explicit
examples where $\theta<1$ at some point.
Consider for instance, in $1D$, $A=(0,1)$ and the function
$u_0(x)=0$ if $x<1/3$ and $x>2/3$, and $2-3x$ if $1/3<x<2/3$.
Then, one shows that $u(t)$ is given by
\[
u(t,x)\ =\begin{cases} 3t & \textup{ if } x<\frac{1}{3}\\
1-2\sqrt{3t}& \textup{ if } \frac{1}{3}<x< 
a(t):= \frac{1}{3}+\frac{2t}{\sqrt{3}}\\
2-3x& \textup{ if } a(t)< x < b(t):=1-\frac{\sqrt{1+6t}}{3}\\
\sqrt{1+6t}-1 & \textup{ if } x> b(t)
\end{cases}
\]
until $t=1-2\sqrt{2}/3$.
We have $\Div u(t)= u(t)_x=(1-2\sqrt{3t}-3t)\delta_{1/3} -3\chi_{(a(t),b(t))}$
for such $t$: $E^+_{u(t)}=\{1/3\}$ stays constant for a while (and disappears
suddenly right after $t=1-2\sqrt{2}/3$), while the density of the
measure $\Div u(t)$ goes down monotonically until it reaches zero
(notice that $v(t)$ will jump right after $1-2\sqrt{2}/3$),
while $E^-_{u(t)}=(\alpha(t),\beta(t))$
shrinks in a continuous way, and carries the constant continuous part
of the initial divergence ($-3$).}
\end{remark}

\begin{proof}
We have $u =u_0 + \delta\nabla v$ with
$-\nabla v \in \partial \D (u)$. Let $x\in \E^+_u$.
Recall that the precise representative of $v$ is defined by
\[
\tilde{v}(x)\ =\ \lim_{\rho\to 0} \frac{\int_{B(x,\rho)}v(y)\,dy}{\omega_N\rho^N},
\]
where $\omega_N=|B(0,1)|$,
and that this limit exists quasi-everywhere in $A$.
We assume also that $\tilde{v}(x)=1$.

Then, for a.e.~$\rho>0$, one may write
\begin{multline}\label{psilepsi0}
\int_{B(x,\rho)} \Div u\,=\,
\int_{\partial B(x,\rho)} u\cdot\nu\,d\H^1
\\
=\,\int_{\partial B(x,\rho)} u_0\cdot\nu\,d\H^1
+\delta \int_{\partial B(x,\rho)}\nabla v\cdot\nu\,d\H^1
\\
=\,\int_{B(x,\rho)}\Div u_0
\,+\,\delta \int_{\partial B(x,\rho)}\nabla v\cdot\nu\,d\H^1.
\end{multline}
Now, let $f(\rho)=(1/\rho^{N-1})\int_{\partial B(x,\rho)} v \,d\H^1$
(which is well-defined for any $\rho$). Then, since $\tilde{v}(x)=1$
and $v\le 1$ a.e.,
\[
\limsup_{\rho\to 0} f(\rho)= N\omega_N\,.
\]
One can also show that for a.e.~$\rho>0$,
$f'(\rho)=(1/\rho^{N-1})\int_{\partial B(x,\rho)} \nabla v\cdot\nu\,d\H^1$,
in fact $f$ is locally $H^1$ in some small interval $(0,\rho_0)$.

Since $v\le 1$ a.e., $f(\rho)\le N\omega_N$ a.e., so that
\begin{multline*}
\liminf_{\e\to 0}
\int_\e^\rho \frac{1}{r^{N-1}}\int_{\partial B(x,r)} \nabla v\cdot\nu\,d\H^1
\,dr
\ =\ \liminf_{\e\to 0}\int_\e^\rho f'(r)\,dr\\
= \ \liminf_{\e\to 0} f(\rho)-f(\e)\ \le\ 0
\end{multline*}
for any $\rho$. If follows that for any $\rho$ small, the set
$I^+_\rho=
\{ r\in [0,\rho]\,:\, \int_{\partial B(x,r)} \nabla v\cdot\nu\,d\H^1\le 0\}$
has positive Lebesgue measure, and for any $r\in I^+_\rho$, we deduce
from~\eqref{psilepsi0} that $\int_{B(x,r)}\Div u\le\int_{B(x,r)}\Div u_0$.

Now consider $\mu$ a positive Radon measure: $\mu$-a.e., we know that the
limits
\[
\frac{\Div u}{\mu}(x)\ =\ \lim_{r\to 0}
 \frac{\int_{B(x,r)}\Div u}{\mu(B(x,r))}
\ \textrm{ and }\ 
\frac{\Div u_0}{\mu}(x)\ =\ \lim_{r\to 0}
 \frac{\int_{B(x,r)}\Div u_0}{\mu(B(x,r))}
\]
exist. If moreover, as before, $x\in \E_u^+$ and $\tilde{v}(x)=1$
(which holds $\mu$-a.e., since $\mu\in H^{-1}(A)$),
we can find a subsequence
$r_n$ such that $\int_{B(x,r_n)}\Div u\le\int_{B(x,r_n)}\Div u_0$
for each $n$, and it follows $(\Div u/\mu)(x)\le(\Div u_0/\mu)(x)$.
\end{proof}

The following corollaries follows:
\begin{corollary}\label{monomeasure}
Let $t>s\ge 0$: then $(\Div u(t))^\pm \le (\Div u(s))^\pm$. In particular,
$\E^\pm_{u(t)}\subseteq \E^\pm_{u(s)}$, $|\Div u(s)|$-a.e.~in $A$.
\end{corollary}
\begin{proof}
Indeed: if $t>s$, then
$u(t)=(I+(t-s)\partial \D)^{-1}(u(s))$. We deduce that for quasi-every
$x\in \E^+_{u(t)}$, $1=\theta_{\Div u(t)}(x)\le (\Div u(s)/(\Div u(t))^+)(x)$,
and it follows
$(\Div u(t))^+\le (\Div u(s)/(\Div u(t))^+) (\Div u(t))^+\le (\Div u(s))^+$.
%If \problem{it should follow the thesis but WHY?}
\end{proof}

\begin{corollary} 
We have that $(\Div u(t))^\pm \wtos (\Div u_0)^\pm$ as $t\to 0$, weakly-$*$
in the sense of measures.
Moreover, $\E^\pm_{u_0}\subset E_r^\pm(0)$ (up to a $|\Div u_0|$-negligible
set), and
$\Div u_0\restr (E_r^+(0))\ge 0$, $\Div u_0\restr (E_r^-(0))\le 0$.
\end{corollary}
\begin{proof}
We know that as $t\to 0$, $u(t)\to u_0$ in $L^2(A;\R^N)$, and
thanks to the boundedness of $\Div u(t)$ it follows that $\Div u(t)\wtos
\Div u_0$ in the sense of measures. Now consider a subsequence $(t_k)$
such that $(\Div u(t_k))^+\wtos \mu$, $(\Div u(t_k))^-\wtos \nu$.
Since $\mu-\nu=\Div u_0$, it follows that $\mu\ge (\Div u_0)^+$
and $\nu \ge (\Div u_0)^-$. The reverse inequalities follow from
Lemma~\ref{lemmeasures} and the first part of the thesis follows.

From the previous results we obtain that for each $t$, one
can write
\[
(\Div u(t))^+\ =\ \theta_t(x)(\Div u_0)^+
\]
The function
$\theta_t(x) = \liminf_{\rho\to 0} (\int_{B(x,\rho)}\Div u(t)^+)/(\int_{B(x,\rho)}
(\Div u_0)^+)$ is well-defined on the set $\E^+_{u_0}$ which supports
the measure $(\Div u_0)^+$,
and we find that $\theta_t(x)\le 1$ is nonincreasing in $t$.
Hence there exists for all
$x\in \E^+_{u_0}$ the limit $\lim_{t\to 0} \theta_t(x)=\sup_{t>0}\theta_t(x)$,
and this limit must be $1$
$(\Div u_0)^+$-a.e., otherwise this would contradict that $(\Div u(t))^+
\wtos (\Div u_0)^+$. It follows that up to a $(\Div u_0)^+$-negligible
set, $\E^+_{u_0}\subseteq \bigcup_{t>0} \{x\in \E^+_{u_0}\,:\, \theta_t(x)>0\}$.

Now, if $x\in \E^+_{u_0}$ and $\theta_t(x)>0$, then $x\in \E^+_{u(t)}$:
indeed,
\[
\frac{\int_{B(x,\rho)}\Div u(t)}{\int_{B(x,\rho)}|\Div u(t)|}
\,=\,
\frac{(\Div u(t))^+(B(x,\rho))-(\Div u(t))^-(B(x,\rho))}
{(\Div u(t))^+(B(x,\rho))+(\Div u(t))^-(B(x,\rho))}
\ \stackrel{\rho\to 0}{\longrightarrow} \ 1\,,
\]
since
\begin{multline*}
(\Div u(t))^-(B(x,\rho))\ \le\  (\Div u_0^-)(B(x,\rho))
\\ =\ o((\Div u_0^+)(B(x,\rho)))\ \le\ o((\Div u(t))^+(B(x,\rho)))
\end{multline*}
 (the equality is because $x\in \E^+_{u_0}$, the last inequality
because $\theta_t(x)>0$). It follows that
\[ \E^+_{u_0} \ \subseteq \ \bigcup_{t>0} \E^+_{u(t)} \]
and the conclusion follows from Lemma~\ref{comparE}.
\end{proof}

%%%%%%%%%%%%%%%%%%%%%%%%%%%%%%%%%%%%%%%%%%%%%%%%%%%%%%
\subsection{The regular case} 
Let us now assume that $\Div u_0=g\in L^p(A)$, $p>1$.
%and $g>0$ on $E^+_0$ while $g<0$ on $E^-_0$.
The obstacle problem which is solved by $w(t)$ can be written
\[
\min_{  w \in H^1_0 : |w|\le t} \half\int_A |\nabla w(x)|^2\,dx\ -\ \int_A g(x)w(x)\,dx\,.
\]

Standard results show that $w(t)\in W^{2,p}(A)$,  (see Theorem  9.9 in \cite{Tru}).
%provided $\partial A$
%is Lipschitz \problem{True? who?}.
In particular, we have that in the $L^p$ sense,
\[
-\Delta w(t) \ =\ g\chi_{\{|w(t)|<t\}}
\]
and, since $u(t)=u_0+\nabla w(t)$, we deduce that in this case
\begin{equation}\label{evoldiv}
\Div u(t)\ =\ \Div u_0\chi_{E(t)}
\end{equation}
for any $t>0$. In particular, formally, we deduce from~\eqref{eqdiv}
that
\begin{equation}\label{evolset}
\Div u_0\frac{\partial \chi_{E(t)}}{\partial t}\ =\ \Delta v(t)\,,
\end{equation}
and since $\Delta v(t)$ is the jump of the normal derivative of
$v(t)$ on $\partial E^\pm(t)$, we find that these sets shrink with
a normal speed $|\nabla v(t)|/|\Div u_0|$. 

This can be written rigorously in the sense of distributions:
$(E^+,E^-,v)$ are such that $v\in L^1 ([0,T);H^1_0(A;[-1,1]))$,
$v=\pm 1$ on $E^\pm$ for a.e.~$t$ and $x$, and
for any $\phi\in C_c^\infty([0,T)\times A)$,
\begin{multline}\label{weakheleshaw}
\int_A \Div u_0(x)\phi(0,x)\,dx
\,+\, \int_0^T\int_A \Div u_0(x)\chi_{E(t)}(x)
\frac{\partial \phi}{\partial t}(x,t)\,dx\,dt
\\ -\ \int_0^T \int_A \nabla v(t,x)\cdot\nabla \phi(t,x)\,dx\,dt\ =\ 0.
\end{multline}
We observe that the evolution equation \eqref{weakheleshaw} is reminiscent of
the enthalpy formulation of the one-phase Stefan problem \cite{rodrigues-long}.

We expect that with either the additional information that $\Div u_0$
is a.e.~nonnegative on $E^+$ and nonpositive on $E^-$,
or that the maps $E^\pm(t)$ are nonincreasing, then \eqref{weakheleshaw}
characterizes the unique evolution~\eqref{eqdiv}. On the other hand,
without this additional assumption, then a time-reversed evolution
with will satisfy the same weak equation, with $u_0$ replaced with
$-u_0$. With {\em both} assumptions we can actually show the following result:

\begin{proposition}
Let $E^+,E^-$ be measurable subsets of $A\times [0,T]$, and
$v\in L^1 ([0,T);H^1_0(A))$ with $|v|\le 1$ a.e., 
$v=\pm 1$ a.e.~on $E^\pm$, and
satisfying~\eqref{weakheleshaw}. Assume in addition that
$\pm\Div u_0\ge 0$ a.e.~on $E^\pm$, and
\begin{equation}\label{monotonE}
E^\pm(t)\ \subseteq \ E^\pm(s) \qquad \text{for a.e. } t>s\,.
\end{equation}
Then $u(t,x) := u_0(x)+\nabla \int_0^tv (s,x)\,ds$ is the
unique solution of~\eqref{eqdiv}.
\end{proposition}
\begin{proof}
Let $w(t)=\int_0^t v(s)\,ds$. Thanks to~\eqref{monotonE}, 
we have that $|w(t,x)|\le t$ for a.e.~$x\in A$, and
$w(t,x)=\pm t$ for a.e.~$x\in E^\pm(t)$, for all $t$.
We can approach test functions of the form $\chi_{[0,t]}\phi(x)$,
$\phi\in H^1_0(A)$, with smooth
functions and pass to the limit to check that
\[
\int_A \Div u_0\phi\,dx-\int_{E(t)}\Div u_0\phi\,dx
\ =\ \int_A \nabla w(t)\cdot\nabla\phi\,dx\,,
\]
for almost all $t$ (up to a negligible set, which we can actually
choose independently of $\phi$, as $H^1_0(A)$ is separable).

If we choose $\phi-w(t,\cdot)$ as the test function in this equation,
we find
\begin{multline*}
\int_A \Div u_0(x)\phi(x)\,dx-\int_A \Div u_0(x) w(x,t)\,dx
-\int_{E(t)}\Div u_0(x)(\phi(x)-w(x,t))\,dx
\\ =\ \int_A \nabla w(t,x)\cdot\nabla\phi(x)\,dx
-\int_A |\nabla w(t,x)|^2\,dx\,,
\\ =\ -\frac{1}{2}\int_A |\nabla w(t,x)-\nabla\phi(x)|^2\,dx+
\frac{1}{2}\int_A |\nabla \phi(x)|^2\,dx
-\frac{1}{2}\int_A |\nabla w(t,x)|^2\,dx\,.
\end{multline*}
If $|\phi|\le t$, we have that $-\Div u_0(x)(\phi(x)-w(x,t))\ge 0$
for a.e.~$x\in E(t)$, so that $w(t)$ is the
minimizer of~\eqref{obst} and the thesis follows.
\end{proof}

\begin{remark}\textup{
As mentioned above,
it is a natural question whether assumption~\eqref{monotonE} is necessary
to prove this result. For instance, in case $E^+$ and $E^-$ are
closed sets in $[0,T)\times A$ with $E^+(t)\cap E^-(t)=\emptyset$ for any $t>0$, 
and $\{ \Div u_0=0\}$ is a negligible set, then one can actually deduce
\eqref{monotonE} from~\eqref{weakheleshaw}.
Indeed, using localized test functions $\phi(x)\chi_{[s,t]}$, one
shows first that $v$ is harmonic in $A\setminus E(t)$ for a.e.~$t$,
and then that $\int_{E(s)}\Div u_0 \phi\,dx- \int_{E(t)}\Div u_0\phi\,dx\ge 0$,
and \eqref{monotonE} follows.
}\end{remark}

% \problem{write the problem with $w$? show that the sets are nonincreasing?
% is it true if they are at least closed ? (I think so: because then
% $v$ is superharmonic near $\partial E^+$ and we find that it must
% decrease)}

%\problem{When is it useful to write what follows? $A$ Lipschitz?}

%We turn now to analyse the regularity properties of *******
\begin{remark}\rm
When $p>N/2$, we can deduce some further properties of $w$
from the regularity theory for the obstacle problem \cite{Caf}.
Indeed, 
letting $\Psi\in H^1_0(A)\cap W^{2,p}(A)$ such that
$-\Delta \Psi=g$, we have that 
$\tilde{w}=w-\Psi\in H^1_0(A)$ solves the obstacle problem
\[
\min_{-t-\Psi\le \tilde{w}\le t-\Psi} \half\int_A |\nabla \tilde{w}(x)|^2\,dx.
\]
Since $p>N/2$, we have $w(t)\in C^{\alpha}(A)$, with $\alpha=2-N/p$,
so that $E(t)=\{|w(t)|=t\}$ is a closed set.
In this case, $v(t)$ can be defined
as the harmonic function in $A\setminus E(t)$ with Dirichlet
boundary condition $v(t)=0$ on $\partial A$ and $v(t)=\pm 1$
on $E^\pm(t)$. Moreover, it is easy to check that $-\nabla v(t)\in \partial^0\D(u(t))$, and $v(t)$
is continuous out of the singular points of $\partial A\cup\partial E(t)$.
%Moreover, if we consider $v_h(t)=(w(t)-w(t-h))/h$, we have $v_h(t)=\pm 1$ in $E^\pm(t)$,
%and $v_h(t)$ is harmonic near any $x\in A\setminus\overline{E(t)}$,
%for $h$ small enough. In particular, the functions $v_h(t)$ converge locally uniformly
%in $A\setminus E(t)$, and $v_h(t,x)\to v(t,x)$ pointwise, as $h\to 0$. 
\end{remark}
\begin{remark}\rm 
If $A=\R^N$ one can easily show easily by a translation argument 
that $u_0\in H^1(A;\R^N)\Rightarrow u(t)\in H^1(A;\R^N)$
with same norm, so that 
the $H^1$-norm of $u(t)$ is nonincreasing.
In this case, $\E_{u(t)}^+$ is a.e.-equivalent to the support of 
$(\Div u)^+$ and since from the equation it follows
$u=u_0$ a.e. on $E^\pm$ (since $v=\pm 1$ a.e. on $E^\pm$, so
that $\nabla v=0$ a.e., the problem being in general that this
will not be true quasi-everywhere), we deduce that $\Div u=\Div u_0$
a.e.~on $E^+\cup E^-=\textup{spt}(\Div u)$.
%I expect this to work also on a convex domain $A$?
\end{remark}

\section{Examples}

\subsection{The antiplane case in dimension 2}\label{secanti}

Let $N=2$ and $k=1$. We have
\[
J(\psi)\ =\ |\rot\psi|(A)\ =\ \sup
\left\{ \int_A \nabla^\perp \cdot\psi\,:\, v\in C_c^\infty(A;[-1,1])
\right\}
\]
where $\rot\psi=\partial_1\psi_2-\partial_2\psi_1$ and 
$\nabla^\perp = (\partial_2,-\partial_1)$. Then, we check easily
that in $L^2(A;\R^2)$ the functional  $J$ is the support function of the closed
convex set
\[
K\ =\ \left\{ \nabla^\perp v\,:\, v\in H^1_0(A;[-1,1])\right\}.
\]
As we mentioned in the Introduction, 
this functional appears as limit of the Ginzburg-Landau model in a suitable energy regime \cite{SS}.

Letting $\psi^\perp = (\psi_2,-\psi_1)$, we
get $J(\psi)=\int_A |\Div \psi^\perp|$, so that the flow can be described as above.

\begin{proposition}\label{propRot}
Let $u_0\in L^2(A;\R^2)$ with $\rot u_0=g\in L^p(A)$, $p>1$. Then for
$t>0$ there exist nonincreasing left-continuous closed (and disjoint)
sets $E^\pm(t)
\subset \{\pm g \ge 0\}$, such that $\rot u(t)=\rot u_0(\chi_{E^-(t)\cup E^+(t)})$.
Moreover, letting $E^\pm=\overline{\cup_{t\ge 0} \{t\}\times E^\pm(t)}$,
there exists a function $v(t,x)$ with $v= \pm 1$ a.e.~on $E^\pm$
such that $(E^+,E^-,v)$ are the unique closed sets and function
solution of the weak Hele-Shaw flow~\eqref{weakheleshaw}.
\end{proposition}

\subsection{The one-dimensional Total Variation Flow}\label{secone}

Let now $N=1$, $k=0$: the previous analysis also provides interesting
qualitative information on the behavior of the flow of the Total
Variation, in dimension 1.

We consider $u_0\in L^2((a,b))$, $a<b$, and the flow $u(t)$ of
the total variation $J(u):=\sup\{\int_a^b u v'\,dt\,:\, v\in C_c^\infty
(a,b;[-1,1])\}$.
Notice that in this situation, the function $w$ which minimizes \eqref{obst},
being in $H^1_0(a,b)$, is also in $C^{1/2}([a,b])$ with $w(a)=w(b)=0$.
In particular, the sets $E^\pm(t)$ defined in~\eqref{defEpm} are closed,
disjoint sets compactly contained in $(a,b)$.

We can state the following result.
\begin{proposition}\label{propTV1}
The function $u(t)$ is the unique minimizer of
\[
\min_{u} J(u)\,+\,\frac{1}{2t} \int_a^b |u-u_0|^2\,dx.
\]
Moreover there exist nonincreasing, disjoint closed sets $E^\pm(t)\subset (a,b)$
such that $u(t)=u_0$ a.e.~on $E^\pm(t)$, $u_0$ is nondecreasing on any
interval contained in $E^+(t)$, nonincreasing on any interval contained
in $E^-(t)$, and $u(t)$ is constant on each connected component
of $(a,b)\setminus (E^+(t)\cup E^-(t))$.
\end{proposition}

If $u_0$ is smooth enough, one can also characterize the speed of the
boundary points of $E^\pm(t)$ in term of $u_0$ and the size of
the intervals of $(a,b)\setminus (E^+(t)\cup E^-(t))$.
\begin{proof}
The first part of the thesis is a consequence of Remark~\ref{remsemigroup}.
Then, if $u_0\in BV(a,b)$, the thesis is a consequence of Lemma~\ref{comparE}.
Indeed, for a.e. $x$ on $E^\pm(t)$, we have
$\partial_x w(t,x)=0$ and $u(t,x)=u_0(x)+\partial_x w(t,x)=u_0(x)$.
If $I\subset E^+(t)$ is an interval, since the measure $Du(t)\restr I$
must be nonnegative, $u(t)$ is nondecreasing on $I$, but as $u(t)=u_0$
a.e.~on $I$ it follows that $u_0$ is nondecreasing on $I$.

If $u_0\not\in BV(a,b)$, we use the fact that for all $\e>0$,
$u(\e)\in BV(a,b)$. Then the Proposition holds for $t>\e$, and
we have $u(t)=u(\e)$ a.e.~on $E^\pm(t)$, $u(\e)$ is nondecreasing on any
interval contained in $E^+(t)$, nonincreasing on any interval contained
in $E^-(t)$, and $u(t)$ is constant on each connected component
of $(a,b)\setminus (E^+(t)\cup E^-(t))$. The sets do not depend
on $\e$, as they are defined as the contact sets in~\eqref{obst}.
Sending then $\e\to 0$ we deduce the result.
\end{proof}

We can deduce the following, quite interesting result
--- see also~\cite{Ring,BoFi,Rybkaetal} for other results on the 
one-dimensional Total Variation flow and in particular~\cite[Prop.~4]{Ring}
for a similar statement.

\begin{corollary}
Let $u_0=\bar u_0+n$ where $\bar u_0\in BV(a,b)$ and 
$n$ is a stochastic process $(a,b)$ with $n\in L^2(a,b)$
a.s.~and such that $|Dn|(I)=+\infty$ for any interval
$I\subset (a,b)$, almost surely.
%% probably any reasonable random noise
%$n$ is a sample of a Gaussian random noise. 
Let $u(t)$ be the total variation flow
starting from $u_0$. Then almost surely, at $t>0$, there is ``staircaising'' 
everywhere in the interval $(a,b)$: $u(t)$ is constant on each connected
component of an open set $A(t)$ which is dense in $(a,b)$.
\end{corollary}
\begin{remark} {\rm The property that $|Dn|(I)=+\infty$ for any interval $I$, almost
surely, is satisfied for instance by the Wiener process (as its quadratic
variation is positive a.s.). For a Gaussian stationary process, it
will depend on the behaviour of the autocorrelation function and can
be characterized by conditions on the power spectrum of the process, see for instance~\cite{Belyaev}
for (non sharp) conditions. %, and % \cite{Doob}
}\end{remark}

\begin{proof}
We let $A(t)=(a,b)\setminus (E^+(t)\cup E^-(t))$, and from the previous
result we know that $u(t)$ is constant on each connected component
of $A(t)$ while $u=u_0$ on $(a,b)\setminus A(t)$. Now assume there
is an interval $I$ with $I\cap A(t)=\emptyset$: without loss of
generality we may assume that $I\subset E^+(t)$. Then $u_0$ must
be nondecreasing on $I$, in particular there exists $I'\subset I$
with $|Du_0|(I')<+\infty$. But this yields that $|Dn|(I')<+\infty$,
which is a.s.~impossible.
\end{proof}

%%%%%%%%%%%%%%%%%%%%%%%%%%%%%%%%%%%%%%%%%%%%%%%%%%%%%%%%%%%%%% 
%\bibliographystyle{plain}
%\bibliography{evol}

\end{document}